\documentclass[12pt]{amsart}
\usepackage[margin=1in]{geometry}

\usepackage{pgfplots}
\usepackage{tikz}
\usepackage{graphicx}
\graphicspath{ {./images/} }
\usepackage{amsmath,comment}
\usepackage{mathrsfs}
\usepackage{amsthm,amssymb}
\usepackage{amsfonts,enumitem}
\usepackage{mathtools}
\graphicspath{ {./images/} }
\usetikzlibrary{shapes}
\usepgfplotslibrary{polar}
\usetikzlibrary{decorations.markings}
\usetikzlibrary{backgrounds}
\pgfplotsset{every axis/.append style={
                    axis x line=middle,    
                    axis y line=middle,    
                    axis line style={<->,color=blue}, 
                    xlabel={$x$},          
                    ylabel={$y$},          
            }}
\usepackage[utf8]{inputenc}

\newtheorem{theorem}{Theorem}
\newtheorem{lemma}[theorem]{Lemma}
\newtheorem{proposition}[theorem]{Proposition}
\newtheorem{corollary}[theorem]{Corollary}

\newtheorem{definition}{Definition}
\newtheorem{question}{Question}
\theoremstyle{remark}

\newtheorem{example}{Example}
\newtheorem{remark}[example]{Remark}

\author{Mike Miller Eismeier}
\address{Department of Mathematics and Statistics, University of Vermont}
\email{Mike.Miller-Eismeier@uvm.edu}
\title{Fourier transforms and integer homology cobordism}
\begin{document}
\maketitle 

\begin{abstract}
We explore the \textit{Fourier transform} of the $d$-invariants, which is particularly well-behaved with respect to connected sum. As corollaries, we show that lens spaces are cancellable in the monoid of 3-manifolds up to integer homology cobordism, and we recover a theorem of Gonz\'alez-Acu\~na--Short on Alexander polynomials of knots with reducible surgeries. 
\end{abstract}

\tableofcontents

\section{Introduction}
The relation of homology cobordism between 3-manifolds has a long and interesting history. Fix a ring $R$. Let $Y$ and $Y'$ be closed oriented 3-manifolds, and suppose $W$ is a compact oriented 4-manifold whose boundary $\partial W$ is oriented diffeomorphic to $Y' - Y$. If both maps \begin{align*}i_*: H_*(Y;R) &\to H_*(W;R) \\ i'_*: H_*(Y';R) &\to H_*(W;R)\end{align*} are isomorphisms, we say that $W$ is an $R$-homology cobordism and that $Y$ and $Y'$ are $R$-homology cobordant. 

This relation is most well-studied when $H_*(Y; R) \cong H_*(S^3; R)$, in which case $Y$ is called an $R$-homology sphere. The set of $R$-homology spheres modulo $R$-homology cobordism form a group $\Theta^3_R$ called the `$R$-homology cobordism group'. The group operation is connected sum, the neutral element is $[S^3]$, and the inverse of $[Y]$ is $[-Y]$. 

If one instead considers the set of all 3-manifolds modulo $R$-homology cobordism, the resulting object is a \emph{monoid}, which we denote by $\widehat \Theta^3_R$; the $R$-homology cobordism group is the sub-monoid of invertible elements.\footnote{If $H_1(Y;R)$ is nonzero, then $Y$ is not invertible in $\widehat \Theta^3_R$; the purported inverse $Y'$ should support an $R$-homology cobordism between $Y \# Y'$ and $S^3$, but $|H_1(Y \# Y';R)| \ge |H_1(Y;R)| > 1 = |H_1(S^3; R)|$.} Some things are known about this monoid; for instance, every equivalence class contains an irreducible 3-manifold \cite{L:hcob-to-irred} or better yet a hyperbolic 3-manifold \cite{M:hcob-to-hyp}. In another direction, there are obstructions to finding a Seifert-fibered manifold in a given equivalence class \cite{CT:no-hcob-to-SF} or more generally to finding a graph manifold whose graph is a tree in a given equivalence class \cite{DH:not-a-tree}.\footnote{Though see \cite{DH-review} for some errata and \cite[Proposition 9.2]{S:new-proof} for a simplified argument.}

In this note, we will investigate integer homology cobordism between a class of 3-manifolds which are not integer homology spheres. In what follows, we suppress the ring $R = \Bbb Z$ from notation, and write integer homology and cohomology groups as $H_*(Y)$ and $H^*(Y)$. 

\begin{theorem}\label{thm:main}
Suppose $L$ and $L'$ are connected sums of lens spaces. If $L$ and $L'$ are integer homology cobordant by a cobordism $W$, then $L$ is oriented diffeomorphic to $L'$, and the induced map $W_*: H_1(L) \to H_1(L')$ respects the natural direct sum decompositions.\footnote{This is meant in an unordered sense. Precisely, suppose $H_1(L) \cong \oplus H_1(L_i)$ and similarly for $L'$. Then $W_*$ sends each $H_1(L_i)$ isomorphically onto some $H_1(L'_j)$.}

Further, if $Y$ is any closed, oriented 3-manifold and $L \# Y$ is integer homology cobordant to $L' \# Y$, then $L$ is oriented diffeomorphic to $L'$.
\end{theorem}

The first part of the result, that the oriented diffeomorphism type of $L$ and $L'$ is determined by their integer homology cobordism type, is not new. It follows from the more general results of \cite{G:lattices} on alternating links, and indeed Greene's results imply that double-branched covers of alternating links are determined by their homology cobordism type. Independent proofs of the more restrictive claim that the oriented diffeomorphism type of a lens space is determined by its $d$-invariants have also appeared \cite{Nem:negdef,DW:lens}. Even before then, the integer homology cobordism classification of lens spaces of odd order goes back to \cite{FS:spherical}. 

Our argument is independent of \cite{G:lattices}, depending only on the computation of Reidemeister torsion for lens spaces and its relationship to their $d$-invariants established in \cite{Nem:negdef}. Theorem \ref{thm:main} is also stronger in two ways: first, it constraints the structure of the span $H_1(L) \leftarrow H_1(W) \to H_1(L')$ of any homology cobordism relating $L$ and $L'$; secondly, it establishes that connected sums of lens spaces are \textit{cancellable} in $\widehat{\Theta}^3_{\Bbb Z}$.


The proof of this theorem is presented in Section 4. The key point is that --- provided a certain non-vanishing property holds --- one can recover (a reduced version of) the $d$-invariants of a connected summand from those of a connected sum.\\

To recover the $d$-invariants of the summands, we find it useful to pass to the \textit{Fourier transform}. Given a function $f: A \to \Bbb C$ on a finite abelian group, its Fourier transform is instead a function on the \emph{dual group} $A^\vee = \text{Hom}(A, S^1)$, defined by $\widehat f(\phi) = \frac{1}{|A|} \sum_{a \in A} f(a) \overline{\phi(a)}$. Here $A$ will be $H^2(Y) \cong H_1(Y)$, and we will pick a base spin$^c$ structure to consider the $d$-invariants as a function on $H_1(Y)$. 

The use of Fourier transforms is well-established in the theory of Reidemeister torsion: analytic interpretations of the Reidemeister torsion (eg \cite{AnTors1, AnTors2}) interpret the torsion as a function of oriented flat bundles, and hence take as input a representation $\phi: H_1(Y;\Bbb Z) \to S^1$. It is also profitable to rephrase the surgery relation in terms of Fourier transforms; Nicolaescu uses this to compare Reidemeister torsion to a Seiberg--Witten invariant in \cite{Nic:SW-QHS}. See also the discussion peppered throughout \cite{Nic:R-book}.

Here we use a simple property of Fourier transforms. Given two groups $A, A'$ and functions $f: A \to \Bbb C$ and $f': A' \to \Bbb C$, the \emph{direct sum} $(f \oplus f')(a,a') = f(a) + f(a')$ on $A \times A'$ has an especially simple Fourier transform; one can effectively read off the values of $\widehat f(\phi)$ and $\widehat f'(\psi)$ from the knowledge of $\widehat{f \oplus f'}$ whenever $\phi$ or $\psi$ are non-trivial homomorphisms. See Proposition \ref{prop:fourier-direct-sum} for a precise statement. 

Applying this purely algebraic observation to $d$-invariants, one can recover (a reduced version of) the $d$-invariants of summands from those of a connected sum. Provided they satisfy a certain non-vanishing property, this recovery process is well-defined up to automorphism of $H^2(Y) \times H^2(Y')$, and thus preserved by integer homology cobordisms. Because the $d$-invariants of lens spaces satisfy this nonvanishing property, and these reduced $d$-invariants --- equivalent to Reidemeister torsion for $L$-spaces --- classify lens spaces up to oriented diffeomorphism, the main theorem follows.\\

In fact, the proof of Theorem \ref{thm:main} exhibits a stronger claim: there exist monoid homomorphisms $c_{p,q}: \widehat{\Theta}^3_{\Bbb Z} \to \Bbb N$ with $c_{p,q}(L(p,q)) = 1$ and $c_{p,q}(L(r,s)) = 0$ unless $L(r,s)$ is oriented diffeomorphic to $L(p,q)$. The purely algebraic part of this claim is the content of Corollary \ref{cor:lin-indep} in Section 2, while the relevant computation for lens spaces is given in Proposition \ref{prop:lens-calc} in Section 4.

If one considers the \emph{Grothendieck group} $\text{Gr}\big(\widehat \Theta^3_{\Bbb Z}\big)$, the group whose elements are pairs $([Y], [Z])$ with $([Y], [Z]) = ([Y'], [Z'])$ if there is an integer homology cobordism $Y \# Z' \sim Y' \# Z$, the existence of these homomorphisms $c_{p,q}$ shows that lens spaces (up to oriented diffeomorphism) span a $\Bbb Z^\infty$ summand of the Grothendieck group.\\

As an aside, the perspective of Fourier transforms appears useful whenever one has an invariant which is additive in the sense above, including both $d$-invariants and Reidemeister torsion. To demonstrate this, in Section 3 we reprove \cite[Theorem 2.2]{its-a-GAS}: if $K$ is a knot with reducible surgery $S_{n}(K) \cong Y \# Y'$ with $H_1(Y) = \Bbb Z/p$ and $H_1(Y') = \Bbb Z/q$, the Alexander polynomial of $K$ is divisible by that of the $(p,q)$ torus knot: $\Delta(T_{p,q}) \mid \Delta(K)$. We hope that Fourier transforms can be a useful organizing tool in other contexts as well.\\

\noindent\textbf{Acknowledgements.} The author would like to thank Danny Ruberman for a useful discussion during the preparation of this note, as well as Tye Lidman for comments on an early draft and suggesting that one might reprove \cite[Theorem 2.2]{its-a-GAS} using the Fourier transform technique. The basic algebraic observation used here came about during conversations with Aiden Sagerman on a related question about products of punctured lens spaces. 

\section{Weighted torsors and Fourier transforms}
In this section we cover the purely algebraic aspects of the main result: the context of weighted torsors, the definition of the Fourier transform, and the process of recovering summands of the direct sum of weighted torsors using Fourier transforms.

\subsection{Weighted torsors}
Here we rapidly define the relevant algebraic objects. Though the initial discussion about torsors is valid for any group, the relevant groups for us will be abelian, so we use additive notation.

\begin{definition}
A \textbf{torsor} is a pair $(A, S)$, where $A$ is a group and $S$ carries a free and transitive (right) action by $A$. 
\end{definition}

A choice of element $s \in S$ gives rise to a bijection $m_s: A \cong S$, given by sending $a \mapsto s + a$; for a different choice of element $s' \in S$ with $s' = s + a_0$, the bijections differ by $(m_{s}^{-1} m_{s'})(a) = a_0 + a.$ 

This observation shows that one may think of a torsor as a group where one has forgotten which element is the identity, or as $A$ `modulo translation'.

\begin{definition}An \textbf{isomorphism of torsors} is a pair $(f, g): (A,S) \to (A', S')$, where $f: A \to A'$ is a group isomorphism and $g: S \to S'$ is a function satisfying $$g(s + a) = g(s) + f(a).$$ 
\end{definition}

By transitivity of the group actions, $g$ is necessarily a bijection. If one chooses basepoints $s \in S$ and $s' \in S'$, and we have $g(s) = s' + a'$, then the map $m_{s'}^{-1} g m_{s}: A \to A$ is given by $$(m_{s'}^{-1} g m_{s})(a) = m_{s'}^{-1} g(s+a) = m_{s'}^{-1}\big(g(s) + f(a)\big) = m_{s'}^{-1}\big(s' + a' + f(a)\big) = a' + f(a).$$ Thinking of $S$ as a group where we've forgotten the identity element (or as a sort of affine space), one should imagine $g$ to be an affine function whose `linear part' is the homomorphism $f$; indeed, one can recover $f$ from $g$. 

\begin{remark}The group of automorphisms of $(A, S)$ is a group sometimes called the \textit{holomorph} of $A$, and can be understood as the group of affine automorphisms of $A$.
\end{remark}

\begin{definition}
A \textbf{weighted torsor} is a torsor $(A, S)$ equipped with a function $d: S \to \Bbb C$. If $S = A$, we call $(A, d)$ a \textbf{weighted group}.
\end{definition}

If one chooses a basepoint $s \in S$, we write $d_s: A \to \Bbb C$ for the function $d_s(a) = (dm_s)(a) = d(s+a)$. For a different choice of basepoint $s' = s+a_0$, we have $d_{s'}(a) = d_s(a+a_0)$. 

\begin{definition}
An \textbf{isomorphism of weighted torsors} $(A,S,d) \to (A',S',d')$ is a torsor isomorphism $(f, g): (A,S) \to (A', S')$ which has $d'(g(s)) = d(s).$
\end{definition}

If one prefers to think entirely in terms of the group $A$ (having chosen an arbitrary basepoint), a weighted torsor is a function $d: A \to \Bbb C$, with $d$ considered equivalent to $d_b(a) = d(a+b)$ for any $b \in A$. In this perspective, an isomorphism of weighted torsors $(A,d) \cong (A',d')$ amounts to an \textit{affine} isomorphism $f+a': A \to A'$ which has $$d'(f(a)+a') = d(a).$$

\subsection{Fourier transforms and weighted duals} Given an abelian group $A$, its \textit{Pontryagin dual} is the group $A^\vee = \text{Hom}(A, S^1)$.\\

\textbf{Convention.} For the rest of this note, abelian groups $A$ and torsors $(A,S)$ are assumed to be \textbf{finite}. This is true in all cases of interest to us, and simplifies discussions of Fourier transforms.\\

Given a function $d: A \to \Bbb C$, we can take its \textit{Fourier transform} $\widehat d: A^\vee \to \Bbb C$, defined as $$\widehat d(\phi) = \frac{1}{|A|} \sum_{a \in A} d(a) \overline{\phi(a)}.$$ 

\begin{remark}\label{rmk:conventions}
This differs from Nicolaescu's definition in \cite[Section 1.6]{Nic:R-book} by a scalar factor of $\frac{1}{|A|}$. Our definition makes some important formulas later slightly simpler.    
\end{remark}

If $d': A \to \Bbb C$ is defined by $d'(a) = d(a+a')$, we have $$\widehat{d'}(\phi) = \frac{1}{|A|} \sum_{a \in A} d'(a) \overline{\phi(a)} = \frac{1}{|A|}\sum_{a \in A} d'(a-a') \overline{\phi(a-a')} = \frac{1}{|A|}\sum_{a \in A} d(a) \overline{\phi(a)} \phi(a') = \widehat d(\phi) \phi(a').$$

This computation inspires the following definition, which we phrase intrinsically on the dual $B = A^\vee$; the statement below implicitly uses the isomorphism $(A^\vee)^\vee \cong A$ for finite $A$. The terminology follows \cite[Definition 3.22]{Nic:R-book} (though notice that Nicolaescu allows for a sign ambiguity, and we do not).

\begin{definition}
If $B$ is an abelian group equipped with weights $\widehat d$ and $\widehat d'$, and there exists some $\psi \in B^\vee$ so that $\widehat d'(b) = \widehat d(b) \psi(b)$ for all $b \in B$, we say that $d$ and $d'$ are \textbf{$t$-equivalent}. We say that weighted groups $(B,\widehat d)$ and $(B', \widehat d')$ are \textbf{$t$-isomorphic} if there exists an isomorphism $f: B \to B'$ and element $\phi \in B^\vee$ so that $$\widehat d'(f(b)) = \widehat d(b) \phi(b)$$ for all $b \in B$. 
\end{definition}

The discussion above shows that given a weighted torsor $(A,S,d)$, choosing a basepoint $s \in S$ and taking the Fourier transform of $d_s$ gives us a weighted group $(A^\vee, \widehat d_s)$, well-defined up to $t$-equivalence. Furthermore, it is clear that isomorphic weighted torsors give rise to $t$-isomorphic weighted groups.\\

Notice that $\widehat d(1) = \sum_{a \in A} d(a)$, and if $\widehat d$ is $t$-isomorphic to $\widehat d'$, then $\widehat d(1) = \widehat d'(1)$. Later, this term will cause us some minor irritation, so we do what we can to remove it.

\begin{definition}\label{def:red}A weighted torsor $(A,S,d)$ is \textbf{reduced} if $\sum_{s \in S} d(s) = 0$; equivalently, $\widehat d(1) = 0$. Given any weighted torsor $(A,S,d)$, its \textbf{reduced part} is given by $(A,S,d^r)$, where $$d^r(s) = d(s) - \frac{1}{|A|} \sum_{s' \in S} d(s');$$ that is, we subtract off the average value. 
\end{definition}

It is clear that $d^r$ is reduced, and that reduction doesn't change the value of the Fourier transform at any $\phi \ne 1$. To see this, we need a small useful lemma.

\begin{lemma}\label{lemma:sum-over-phi}
If $A$ is a finite abelian group and $\phi: A \to S^1$ is a homomorphism, we have $$\sum_{a \in A} \phi(a) = \begin{cases} 0 & \phi \ne 1 \\
|A| & \phi = 1 \end{cases}$$
\end{lemma}
\begin{proof}
This is a special case of the orthogonality relations for irreducible characters \cite[Theorem 2.3.3]{serre:rep}. The proof is included for completeness. 

If $\phi$ is trivial this is obvious. For $\phi$ nontrivial, write $\zeta$ for a generator of $\phi(A)$ so that $\zeta$ is a primitive $m$th root of unity for some $m > 1$. We have $$\sum_{a \in A} \phi(a) = \frac{|A|}{m} \sum_{k=0}^{m-1} \zeta^k.$$ But $\sum_{k=0}^{m-1} \zeta^k = \frac{1-\zeta^m}{1-\zeta} = 0$ for $\zeta \ne 1$ a non-trivial $m$th root of unity.
\end{proof}

It follows that if two weighted torsors have $d'(a) = d(a) + c$ for all $a \in A$ and some constant $c$, we have $$\widehat{d'}(\phi) = \frac{1}{|A|}\sum_{a \in A} d(a) \overline{\phi(a)} + \frac{1}{|A|}\sum_{a \in A} c \overline{\phi(a)} = \begin{cases} \widehat d(1) + c & \phi = 1 \\ \widehat d(\phi) & \phi \ne 1 \end{cases}$$ 

\begin{corollary}
If $(A,S,d)$ is a weighted torsor, its reduced part $(A, S, d^r)$ satisfies $$\widehat{d^r}(\phi) = \begin{cases} 0 & \phi = 1 \\ \widehat d(\phi) & \phi \ne 1 \end{cases}$$
\end{corollary}

\subsection{Direct sums of weighted torsors}
Given two weighted torsors $(A, S, d)$ and $(A', S', d')$ we say their \textbf{direct sum} is the weighted torsor $(A \times A', S \times S', d \oplus d')$, where $$(d \oplus d')(s,s') = d(s) + d(s').$$ 

Notice that the $(A \times A')^\vee$ is naturally isomorphic to $A^\vee \times (A')^\vee$; if $\phi: A \to S^1$ and $\psi: A' \to S^1$ are homomorphisms, these give rise to the homomorphism $\phi \psi: A \times A' \to S^1$ by pointwise multiplication: $(\phi \psi)(a,a') = \phi(a) \psi(a')$.\\

The observation which motivated the present note is the following calculation of the Fourier transform of a direct sum of weighted torsors. 

\begin{proposition}\label{prop:fourier-direct-sum}
If $d \oplus d': A \times A' \to \Bbb Q$ is the direct sum of two weighted groups, the Fourier transform satisfies $$\widehat{(d \oplus d')}(\phi \psi) = \begin{cases} \widehat d(\phi) & \phi \ne 1 \text{ and } \psi = 1 \\ \widehat d'(\psi) & \phi = 1 \text{ and } \psi \ne 1 \\ \widehat d(1) + \widehat d'(1) & \phi = \psi = 1 \\ 0 & \text{else} \end{cases}$$
\end{proposition}

\begin{proof}
We have \begin{align*}\widehat{(d \oplus d')}(\phi\psi) &= \frac{1}{|A||A'|}\sum_{a,a'} (d \oplus d')(a, a') \overline{\phi\psi(a,a')} = \frac{1}{|A||A'|}\sum_{a,a'} \big(d(a) + d'(a')\big) \overline{\phi(a)} \overline{\psi(a')} \\
&= \frac{1}{|A||A'|}\left(\sum_{a,a'} d(a) \overline{\phi(a)} \overline{\psi(a')}\right) + \frac{1}{|A||A'|}\left(\sum_{a,a'} d'(a') \overline{\psi(a')} \overline{\phi(a)}\right) \\
&= \left(\frac{1}{|A|}\sum_a d(a) \overline{\phi(a)}\right) \left(\frac{1}{|A'|}\sum_{a'} \overline{\psi}(a')\right) + \left(\frac{1}{|A'|}\sum_{a'} d'(a') \overline{\psi(a')}\right) \left(\frac{1}{|A|}\sum_a \overline{\phi}(a)\right).
\end{align*}

By Lemma \ref{lemma:sum-over-phi}, the first term vanishes when $\psi \ne 1$ and is $|A'|\widehat d(\phi)$ when $\psi = 1$, while the second term vanishes when $\phi \ne 1$ and is $\widehat d'(\psi)$ when $\phi = 1$. This gives the stated claim in all cases except $\phi, \psi = 1$; in that case, it gives $\widehat d(1) + \widehat d'(1)$.
\end{proof}

In particular, for \textbf{nontrivial} $\phi$ and $\psi$, one can read off the values of $\widehat d(\phi)$ and $\widehat d'(\psi)$ from the Fourier transform of the direct sum $\widehat{d \oplus d'}$. In the non-reduced case, the fact that $\widehat d(1)$ and $\widehat d'(1)$ are combined in $\widehat{(d \oplus d')}(1)$ means that we cannot recover this information from the Fourier transform of the direct sum. This is why we restrict attention to reduced weighted torsors.

\subsection{Nonvanishing properties and recovering invariants of summands}
We will now be more precise about the process of recovering the summands of a direct sum of weighted torsors in a way which is well-defined up to isomorphism. To do so, we must make some further assumptions. 

\begin{definition}A weighted group $(B, \widehat d)$ has the \textbf{nonvanishing property} if $\widehat d(b) \ne 0$ for all \textbf{nontrivial} elements $b \in B$. 
\end{definition}

Notice that this property is well-defined up to $t$-isomorphism, because if $(B,\widehat d)$ and $(B',\widehat d')$ are $t$-isomorphic, we have an isomorphism $f: B \to B'$, an element $\phi \in B^\vee$, and an equality $\widehat d'(f(b)) = \widehat d(b) \phi(b)$. Because $b$ is non-trivial if and only if $f(b)$ is, and $\phi(b) \in S^1$ is nonzero, $\widehat d'$ has the nonvanishing property if and only if $\widehat d$ does.

\begin{definition}
If $(B, \widehat d)$ is a weighted group, a \textbf{special subgroup} is a \textbf{non-trivial} subgroup $C \subset B$ so that $\widehat d(c) \ne 0$ for all non-trivial $c \in C$. A \textbf{maximal special subgroup} is a special subgroup which is maximal among special subgroups.
\end{definition}

Notice that special subgroups are well-defined up to $t$-equivalence of weights, and that $t$-isomorphism preserves maximal special subgroups: if $f: (B, \widehat d) \to (B', \widehat d')$ is a $t$-isomorphism and $C \subset B$ is a special subgroup, then $f(C)$ is too, and vice versa. Further, a $t$-isomorphism maps $(C, \widehat d|_C)$ $t$-isomorphically onto $(f(C), \widehat d'|_{f(C)})$. In particular, the maximal special subgroups (considered as weighted groups up to $t$-isomorphism) are $t$-isomorphism invariants of $(B, \widehat d)$.

It immediately follows from this that isomorphisms between direct sums of weighted torsors with the nonvanishing property are rather constrained.

\begin{corollary}\label{cor:decomp}
Suppose $\{(A_i, d_i)\}_{i=1}^n $ is a collection of weighted groups whose Fourier transforms satisfy the nonvanishing property, and similarly with $\{(A'_j, d'_j)\}_{j=1}^m$. Write $(A, d) = \bigoplus_{i=1}^n (A_i, d_i)$ and similarly for $(A', d')$. If there is an affine isomorphism of weighted groups $\varphi+a': (A, d) \cong (A', d')$, then $n = m$ and the map $\varphi$ preserves the direct sum decompositions, in the sense that for all $i$ we have $\varphi(A_i) = A'_j$ for some $j$.
\end{corollary}
\begin{proof}
By Proposition \ref{prop:fourier-direct-sum}, the maximal special subgroups of $A^\vee = A_1^\vee \oplus \cdots \oplus A_n^\vee$ are precisely the coordinate axes $A_i^\vee$ (where all coordinates but the $i$th are nonzero). Comparing the number of maximal special subgroups, we see that $m = n$. Because $\varphi^\vee$ maps maximal special subgroups bijectively to maximal special subgroups, for some permutation $\sigma$ we have $\varphi^\vee((A'_{\sigma(i)})^\vee) = A_i^\vee$ for all $i$. That is, if $\psi': A' \to S^1$ is any homomorphism, then $\psi'$ factors through $\pi'_{\sigma(i)}: A' \to A'_{\sigma(i)}$ if and only if there exists some $\psi: A_i \to S^1$ with $$\psi \pi_i = \psi' \pi'_{\sigma(i)} \varphi.$$

It will follow that $\varphi(A_i) \subset A'_{\sigma(i)}$, and then equality follows because these have the same cardinality (their duals do) and $\varphi$ is injective. To see this first claim, pick $x_i \in A_i$, and consider $y_j = \pi_j \varphi(x_i)$. If $y_j$ is nonzero, there is some homomorphism $\psi'_j: A'_j \to S^1$ with $\psi'_j(y_j) \ne 0$. By the discussion above, $$\psi'_j(y_j) = \psi'_j \pi'_j \varphi(x_i) = \psi \pi_{\sigma^{-1}(j)}(x_i).$$ This can only be nonzero if $\sigma^{-1}(j) = i$ by assumption, so that $\varphi(x_i)$ indeed lies in $A'_{\sigma(i)}$.
\end{proof}

We will prove the main theorem similarly, by counting maximal special subgroups (considered as weighted groups up to $t$-isomorphism); we introduce notation for this special concept.

\begin{definition}\label{def:MS}
Given a weighted group $(B, \widehat d)$, we associate the multiset $$MS(B, \widehat d) = \left\{\left[C, \widehat d|_C\right] \; \bigg| \; C \subset B \textup{ is a maximal special subgroup}\right\}$$ of special subgroups of $B$ equipped with the restriction of $\widehat d$, considered up to $t$-isomorphism.
\end{definition}

Recall here that a multiset $M$ is a set (by an abuse of notation written with the same name $M$) where each element $x \in M$ is equipped with a weight $w_M(x) \ge 1$ labeling how many times it occurs in the multiset. 

Notice that if $(B, \widehat d)$ is $t$-isomorphic to $(B', \widehat d')$, then the multisets $MS(B, \widehat d)$ and $MS(B', \widehat d')$ are isomorphic (there is a weight-preserving bijection between them). If $(A, S, d)$ is a weighted torsor, the multiset $MS(A^\vee, \widehat d_s)$ is an invariant of $(A, S, d)$; isomorphic weighted torsors give rise to isomorphic multisets. We write $MS(A, S, d)$ for $MS(A^\vee, \widehat d_s)$ for some choice of $s \in S$.\\

If $M, N$ are multisets, we write $M \cup N$ for the multiset whose underlying set is the union of the underlying sets of $M$ and $N$, and whose weight is $w_{M \cup N}(x) = w_M(x) + w_N(x)$. (Here we write $w_N(x) = 0$ if $x$ does not lie in $N$, and similarly with $M$.)

The crucial observation, almost immediate from Proposition \ref{prop:fourier-direct-sum}, is that this multiset is additive. 

\begin{proposition}\label{prop:additivty}
If $(A, S, d)$ and $(A', S', d')$ are \textbf{reduced} weighted torsors, we have $$MS(A \times A', S \times S', d \oplus d') = MS(A, S, d) \cup MS(A', S', d').$$
\end{proposition}

\begin{proof}
For convenience, we write $d^{\oplus} = (d \oplus d')$, and $\widehat d^{\oplus}$ for its Fourier transform.\\

As mentioned above, Proposition \ref{prop:fourier-direct-sum} implies that the maximal special subgroups of $(A \times A')^\vee \cong A^\vee \times (A')^\vee$ are precisely the maximal special subgroups of $A^\vee \times \{1\}$ and $\{1\} \times (A')^\vee$. 

Given a maximal special subgroup $C \subset A^\vee$ (or similarly $C' \subset (A')^\vee$), what remains is to compare the restriction of $\widehat d$ to $C$ with the restriction of $\widehat{d^\oplus}$ to $C \times \{1\}$, but $$\widehat d^\oplus|_{C \times \{1\}} = \widehat d|_C$$ by the formula from Proposition \ref{prop:fourier-direct-sum} and the assumption that $d$ and $d'$ are \textit{reduced} weighted torsors. 
\end{proof}

We can now define monoid homomorphisms from the appropriate monoid to $\Bbb N$.

\begin{definition}
We write $\widehat\Theta_{\textup{WT}}$ for the monoid whose elements are weighted torsors up to isomorphism, and whose product operation is direct sum.
\end{definition}

There is a corresponding monoid $\widehat\Theta_{\textup{RWT}}$ of reduced weighted torsors, and the map $d \mapsto (d^r, \text{avg } d)$ defines a monoid isomorphism $\widehat\Theta_{\textup{WT}} \cong \widehat\Theta_{\textup{RWT}} \times \Bbb C$.

Write $\textup{RWTN}$ for the set of reduced weighted torsors whose Fourier transforms satisfy the nonvanishing property, considered up to isomorphism; because these are considered up to isomorphism, we may think of these as weighted groups up to affine isomorphism and drop the torsor $S$ from notation. 

For each $[A, d] \in \textup{RWTN}$, we define a map $c_{A,d}: \widehat\Theta_{\textup{WT}} \to \Bbb N$ with $$c_{[A,d]}(B,T,f) = \# \text{ of occurrences of } [A^\vee,\widehat{d}] \text{ in } MS(B,T,f^r).$$ That is, $c_{[A,d]}(B,T,f)$ is the weight $w_{MS(B,T,f^r)}([A^\vee,\widehat d])$.

\begin{corollary}\label{cor:lin-indep}
The functions $c_{[A,d]}$ are monoid homomorphisms. If $(A', d')$ is another reduced weighted torsor whose Fourier transform has the nonvanishing property, we have $$c_{[A,d]}(A',d') = \begin{cases} 1 & (A,d) \text{ is isomorphic to } (A',d') \\ 0 & \text{else} \end{cases}$$
\end{corollary}
\begin{proof}
A maximal special subgroup is defined to be non-trivial, so for the trivial weighted torsor with underlying group $1$ and zero weighting, $MS(1,0) = \varnothing$; so $c_{[A,d]}(1,0) = 0$ and thus $c$ sends neutral element to neutral element. Additivity follows immediately from Proposition \ref{prop:additivty} and the fact that taking the reduced part $d \mapsto d^r$ commutes with direct sums. So $c_{[A,d]}$ is a monoid homomorphism.

Because the Fourier transform of $(A',d')$ has the nonvanishing property, $MS(A,d) = \{[(A')^\vee, \widehat{d'}]\}$. If $[A^\vee, \widehat{d}]$ appears in this singleton set, then in fact $\big((A')^\vee, \widehat{d'}\big)$ is $t$-isomorphic to $(A^\vee, \widehat d)$, and hence $(A,d) \cong (A', d')$.
\end{proof}

It follows that the functions $c$ assemble into a surjective monoid homomorphism $c: \widehat\Theta_{\textup{WT}} \to \Bbb N^{\textup{RWTN}}$, which behaves particularly well on reduced weighted torsors whose Fourier transforms have the nonvanishing property: there is a map $\Bbb N^{\textup{RWTN}} \to \widehat\Theta_{\textup{WTN}}$ whose composition with $c$ is the identity.

\section{A theorem of Gonz\'alez-Acu\~na--Short}
Before moving on to the main theorem, we use this opportunity to give an alternative proof of \cite[Theorem 2.2]{its-a-GAS}, suggested to the author by Tye Lidman. 

For context, if $\Sigma$ is a homology sphere and $K = C_{p,q}(K')$ is the cable of another knot $K' \subset \Sigma$, then the $pq$-surgery satisfies $\Sigma_{pq}(K) \cong \Sigma_{p/q}(K') \# L(q,p)$. When $\Sigma = S^3$, the cabling conjecture \cite[Conjecture A]{its-a-GAS} predicts that this construction gives the \textit{only examples} of knots with reducible surgery. Among their evidence was the following theorem. 

\begin{theorem}
Let $K \subset \Sigma$ be a knot in an integer homology sphere with reducible surgery $\Sigma_{n/m}(K) \cong Y_1 \# Y_2$, where $|H_1(Y_1)| = p>1$ and $|H_1(Y_2)| = q>1$. Then the polynomial $\Delta_{p,q} = \frac{(t^{pq}-1)(t-1)}{(t^p-1)(t^q-1)}$ divides the Alexander polynomial $\Delta_K$. 
\end{theorem}

Note that the Alexander polynomial of a cable knot $K = C_{p,q}(K')$ satisfies $\Delta_K = \Delta_{p,q} \Delta_{K'}$.  

\begin{proof}
The proof makes use of the Reidemeister torsion of a 3-manifold, and we quickly recall some properties from \cite[Section 3.7]{Nic:R-book}. When $Y$ is a rational homology sphere, its Reidemeister torsion $T_Y: H_1(Y) \to \Bbb Q$ makes $H_1(Y)$ into a weighted torsor, well-defined up to isomorphism and multiplication by $\pm 1$. When $H_1(Y) \cong \Bbb Z$, by contrast, $T_Y(t)$ should be understood as a rational function $T_Y \in \Bbb Q(t)$, well-defined up to  multiplication by $\pm t^k$. The Fourier transform $\hat T_Y: \Bbb C \to \Bbb C$ is a meromorphic function defined by evaluating $T_Y$ on a given complex number, and is well-defined up to a variation on $t$-equivalence: $\hat T'(z) \sim \pm z^k \hat T(z)$. 

In the case of knot complements, the Reidemeister torsion is related to the Alexander polynomial by the formula $$T_{\Sigma \setminus K}(t) = \pm \frac{t^k \Delta_K(t)}{1-t};$$ this first appeared as \cite[Theorem 4]{Milnor:R-Alex}. 

We will use the surgery formula for the Fourier-transformed Reidemeister torsion as stated in \cite[Theorem 3.23]{Nic:R-book}: if $\zeta$ is a primitive $n$th root of unity, then $$\widehat{T}_{\Sigma_{n/m}(K)}(\zeta) = \frac{\widehat T_{\Sigma \setminus K}(\zeta)}{(1-\zeta)^{-1}} = \pm \frac{\zeta^k \Delta_K(\zeta)}{(1-\zeta^{-1})(1-\zeta)}.$$ In particular, the zeroes of $\widehat T_{\Sigma_{n/m}(K)}$ are identified with the $n$th roots of unity $\zeta$ for which $\Delta_K(\zeta) = 0$. 

Here we use the canonical isomorphism $H_1\big(\Sigma_{n/m}(K)\big) \cong \Bbb Z/n$, sending a meridian of the knot in $\Sigma \setminus K$ to $1$ to identify $H_1^\vee$ with the group of $n$th roots of unity.

Suppose $K$ is a knot as in the statement of the theorem. The isomorphism $\Bbb Z/n \cong \Bbb Z/p \times \Bbb Z/q$ induced by the connected sum decomposition sends the elements $(i,j)$ with $i,j$ nontrivial to $n$th roots of unity which are neither $p$th nor $q$th roots of unity. For rational homology spheres $Y_1$ and $Y_2$ we have $T_{Y_1 \# Y_2}(i,j) = T_{Y_1}(i) + T_{Y_2}(j)$ \cite[Theorem XII.1.2]{Tur:R-book}. It follows from Proposition \ref{prop:fourier-direct-sum} that $\widehat T_{Y_1 \# Y_2}(\zeta) = 0$ for any $n$th root of unity $\zeta$ which is neither a $p$th nor $q$th root of unity. Thus $\Delta_K(\zeta) = 0$ for all such roots of unity, so that $\Delta_K$ is divisible by $$\prod_{\substack{\zeta = e^{2\pi i k/n} \\ 0 \le k < n \\ p \nmid k \text{ and } q \nmid k}} (t-\zeta) = \frac{(t^{pq}-1)(t-1)}{(t^p-1)(t^q-1)} = \Delta_{p,q}$$ as claimed.
\end{proof}

\section{$d$-invariants of 3-manifolds}
If $Y$ is a 3-manifold, there is a naturally associated torsor $\big(H^2(Y)_{\text{tors}}; \text{Spin}^c(Y)_{\text{tors}}\big)$, where the latter is the set of spin$^c$ structures with torsion first Chern class. When $Y$ is a rational homology sphere, every spin$^c$ structure is torsion.

When we refer to a homology cobordism, we mean a pair $(W, \varphi)$ of a compact oriented connected 4-manifold and a \textbf{chosen} orientation-preserving diffeomorphism $\varphi: \partial W \cong Y - Y'$ so that the corresponding maps $Y \to W$ and $Y' \to W$ induce isomorphisms on all integer cohomology groups.

Given a homology cobordism $W: Y \to Y'$ there is an induced isomorphism of torsors $$(W_*, W^c_*): \big(H^2(Y)_{\text{tors}}, \text{Spin}_{\text{tors}}^c(Y)\big) \to \big(H^2(Y')_{\text{tors}}; \text{Spin}_{\text{tors}}^c(Y')\big).$$

We will make use of three weighted torsors; the first and third are associated to rational homology spheres, while the second is associated to an arbitrary 3-manifold.

\begin{itemize}
    \item The \emph{$d$-invariant}, $d_Y: \text{Spin}^c(Y) \to \Bbb Q$ \cite[Definition 4.1]{OS:d};
    \item The \emph{twisted $d$-invariant}, $\underline{d}_Y: \text{Spin}^c_{\text{tors}}(Y) \to \Bbb Q$ \cite[Definition 3.1]{BG:twisted-d}; 
    \item The \emph{Turaev-Reidemeister torsion}, $T_Y: \text{Spin}^c(Y) \to \Bbb Q$ \cite[Chapter X]{Tur:R-book}.
\end{itemize}

\begin{remark}
The Turaev-Reidemeister torsion is often written as an $H^2(Y)$-equivariant map $\text{Spin}^c(Y) \to \Bbb Q[H^2(Y)]$, eg \cite[Chapter I.4.1]{Tur:R-book}. This gives rise to the function $T_Y$ above by extracting the coefficient of $0 \in H^2(Y)$. This can be extended to an arbitrary 3-manifold, but the discussion is somewhat more intricate when $H^2(Y)$ is infinite: instead, the torsion defines an $H^2(Y)$-equivariant map to the fraction field $\Bbb Q\big(H^2(Y)\big)$.
\end{remark}

The twisted $d$-invariant is only used for a technical reason, to allow connected sums with arbitrary 3-manifolds instead of merely rational homology spheres. When $Y$ is a rational homology sphere, we have the tautological equality $\underline d_Y(\mathfrak s) = d_Y(\mathfrak s)$. The Turaev-Reidemeister torsion --- and the relation to $d$-invariants --- will be used exclusively for calculation.

First, we establish the relationship to the work from Section 2.

\begin{lemma}
The assignment $Y \mapsto (H^2(Y)_{\textup{tors}},\; \textup{Spin}^c(Y)_{\textup{tors}},\; \underline d_Y)$ defines a monoid homomorphism $\widehat \Theta_{\Bbb Z} \to \widehat \Theta_{\textup{WT}}$.
\end{lemma}
\begin{proof}
This amounts to three claims: that $\underline d_Y(S^3) = 0$ (tautological), that the assignment $Y \mapsto \underline d_Y$ sends integer homology cobordisms to isomorphisms of weighted torsors (an immediate corollary of \cite[Corollary 4.2]{BG:twisted-d}), and that $\underline d_Y$ is additive, in the sense that $$\underline d_{Y \# Y'}(\mathfrak s \# \mathfrak s') = \underline d_Y(\mathfrak s) + \underline d_Y(\mathfrak s'),$$ 
which is \cite[Proposition 3.7]{BG:twisted-d}.
\end{proof}

From this, we can immediately show that \textit{provided the $\widehat d$-invariants of the summands satisfy the nonvanishing property,} integer homology cobordisms between connected sums preserve the natural direct sum decomposition of their homology groups.

\begin{corollary}\label{cor:cob-decomp}
Suppose $Y = \#_{i=1}^n Y_i$ and $Y' = \#_{j=1}^m Y'_j$ are connected sums of 3-manifolds so that the $\widehat{\underline d}$-invariants of each $Y_i$ and $Y'_j$ satisfy the nonvanishing property. If $W: Y \to Y'$ is a homology cobordism, then $n = m$ and the induced map $W_*: H_1(Y)_{\textup{tors}} \to H_1(Y')_{\textup{tors}}$ preserves the natural (unordered) direct sum decompositions.
\end{corollary}
\begin{proof}
As mentioned above, if $W: Y \to Y'$ is a homology cobordism, it induces an isomorphism of weighted torsors $(\text{Spin}^c_{\text{tors}}(Y), \underline d) \cong (\text{Spin}^c_{\text{tors}}(Y'), \underline d')$. The statement follows immediately from Corollary \ref{cor:decomp}.
\end{proof}

The following corollary is simply an application of Corollary \ref{cor:lin-indep}, applied to these particular weighted torsors.

\begin{corollary}\label{cor:generate-free}
Suppose $Y_i$ is a collection of 3-manifolds indexed by some set $S$ with the following properties: \begin{itemize}
    \item The groups $H^2(Y_i)$ are non-trivial.
    \item The Fourier transforms $\widehat{\underline d}_{Y_i}$ satisfy the nonvanishing property.
    \item The weighted torsors $(H^2(Y_i),\; \textup{Spin}^c(Y_i)_{\textup{tors}},\; \underline d^r)$ are pairwise non-isomorphic.
\end{itemize}
Then there is a homomorphism $c: \widehat \Theta_{\Bbb Z} \to \Bbb N^S$ with $c_i(Y_i) = 1$ and $c_i(Y_j) = 0$ for $i \ne j$. In particular, the $Y_i$ are linearly independent in $\widehat \Theta_{\Bbb Z}$ and span a $\Bbb Z^S$-summand of the Grothendieck group $\textup{Gr}(\widehat \Theta_{\Bbb Z})$.
\end{corollary}

Here $\underline d^r$ is the reduced part of $\underline d$ as in Definition \ref{def:red}.

To prove Theorem \ref{thm:main}, we need to show that the lens spaces $L(p,q)$ --- considered up to oriented diffeomorphism --- satisfy the assumptions of the corollary. This is classical; the crucial observation is that the reduced $d$-invariant recovers the Turaev-Reidemeister torsion.

\begin{lemma}
If $Y$ is an $L$-space, we have $$T(\mathfrak s) = \frac{1}{2}\left(\underline d^r(\mathfrak s)\right).$$
\end{lemma}

This follows immediately from \cite[Theorem 5.3.3-4]{Rustamov:thesis}. For lens spaces (and thus their connected sums, as both sides of this equality are additive) this was proven earlier as \cite[Section 10.7]{Nem:negdef}, and indeed N\'emethi's result is used in Rustamov's argument.

The Turaev-Reidemeister torsion of lens spaces is classical, and the sign-refined version only slightly less so. For an appropriate choice of base spin$^c$ structure and an appropriate isomorphism $H^2(L(p,q)) \cong \Bbb Z/p$, we have \cite[Section 7.1]{NS:sw} for any nontrivial $p$th root of unity $\zeta$ $$\widehat T(\zeta) = \frac{1}{p(1-\zeta^{-1})(1-\zeta^{-q})}.$$

Notice that this formula has an extra factor of $1/p$ compared to Nicolaescu's, owing to the change of convention discussed in Remark \ref{rmk:conventions}.

We write $d^r_{p,q}: \text{Spin}^c(L(p,q)) \to \Bbb Q$ for the reduced $d$-invariants of the lens space $L(p,q)$.

The lemma above establishes that (after choosing an appropriate base spin$^c$ structure and isomorphism $L(p,q) \cong \Bbb Z/p$) we have for all nontrivial $p$th roots of unity $$(\widehat d^r_{p,q})(\zeta) = \frac{1}{2p(1-\zeta^{-1})(1-\zeta^{-q})}.$$

\begin{proposition}\label{prop:lens-calc}
The Fourier transforms of the weighted torsors $d^r_{p,q}$ have the nonvanishing property. Furthermore, if $d^r_{p,q} \cong d^r_{r,s}$, then $p = r$ and $s \equiv q^{\pm 1} \mod p$, so that $L(p,q)$ and $L(r,s)$ are oriented diffeomorphic.
\end{proposition}

\begin{proof}
That these have the nonvanishing property is obvious.

What remains is the essentially classical claim that signed Turaev-Reidemeister torsion classifies lens spaces up to oriented diffeomorphism; we include a short proof for completeness. We may as well assume $p > 2$.

If $d^r_{p,q} \cong d^r_{r,s}$ we have $\Bbb Z/p \cong \Bbb Z/r$, so that $p = r$. If $\widehat d^r_{p,q}$ is $t$-isomorphic to $\widehat d^r_{p,s}$, such an isomorphism induces a $t$-isomorphism $\widehat f_{p,q} \cong \widehat f_{p,s}$ between the simpler functions $$\widehat f_{p,q}(\zeta) = (1-\zeta)(1-\zeta^q) = 1 - \zeta - \zeta^{q} + \zeta^{q+1}.$$

If $q = 1$ this is the Fourier transform of the function $f_{p,1} = (1, -2, 1, 0, \cdots, 0)$, where we list off the values $f_{p,q}(i)$ in order starting at $0$. If $q = p-1$ we have instead $f_{p,p-1} = (2, -1, 0, \cdots, 0, -1)$. For $1 < q < p-1$, we have $f_{p,q}(0) = f_{p,q}(q+1) = 1$, while $f_{p,q}(1) = f_{p,q}(q) = -1$, and all other values are zero.

It is transparent that there is no affine isomorphism of $\Bbb Z/p$ taking $f_{p,1}$ or $f_{p,p-1}$ to any of the other functions above, as the values are different; this reduces us to the case $1 < q < p-1$.\\

Now suppose there exists some integer $k$ prime to $p$ and some integer $\ell$ so that \begin{equation}\label{eqn-equiv}f_{p,s}(ki+\ell) = f_{p,q}(i)\end{equation} for all $i$. Because $f_{p,q}(0) = 1$, we have $f_{p,s}(\ell) = 0$ and thus either $\ell \equiv 0$ or $\ell \equiv s+1$. Because $i \mapsto ki+\ell$ is a bijection, in the former case we must have $k(q+1) \equiv s+1$ and in the latter case $k(q+1) + s+1 \equiv 0$. We handle these two cases separately.\\
\begin{itemize}
    \item \textbf{Case 1: $\ell \equiv 0$.} Applying (\ref{eqn-equiv}) to $i \equiv 1$, we have either $k \equiv 1$ (which gives $q + 1 \equiv s+1$ and hence $q \equiv s$) or $k \equiv s$ (which gives $qs+s \equiv s+1$ so $qs \equiv 1$).\\
    \item \textbf{Case 2: $\ell \equiv s+1$.} Applying (\ref{eqn-equiv}) to $i\equiv1$, we have either $k+s+1 \equiv 1$ (in which case $k \equiv -s$ so $-sq-s + s +1 \equiv 0$ and thus $qs \equiv 1$) or $k+s+1 \equiv s$ (in which case $k \equiv -1$ so that $-q-1+s+1 \equiv 0$ and $q \equiv s$). \\
\end{itemize}
In any of the four possibilities for the values of $k,\ell$ modulo $p$, we see that the desired claim holds.
\end{proof}

The main theorem follows immediately from this proposition, as well as Corollaries \ref{cor:cob-decomp}-\ref{cor:generate-free}.

\section{Questions}
We close with a handful of questions inspired by the results above. 

\begin{question}
Which collections of 3-manifolds satisfy the hypotheses of Corollary \ref{cor:generate-free}? Does this class include spherical 3-manifolds, or double-branched covers of alternating links?
\end{question}

It would follow that manifolds in this class are integer homology cobordant if and only if they are diffeomorphic. For spherical 3-manifolds, an argument might proceed by an explicit computation of their $d$-invariants (or of their Reidemeister torsions); for alternating links an argument might follow through the lattice-theoretic techniques of \cite{G:lattices}.\\

In another direction, recall that \cite{Lisca:Lens2} provides a complete classification of connected sums of lens spaces up to rational homology cobordism. It would be interesting to determine the intermediate case of $R$-homology cobordism for, say, $R = \Bbb Z/p$.

\begin{question}
Determine the $R$-homology cobordism classification of connected sums of lens spaces for various rings $R$.
\end{question}

The author has made no attempt to investigate this. It would similarly be interesting to follow up on the classification of rational homology ribbon cobordisms between connected sum of lens spaces given in \cite[Theorem 1.3]{Huber} and classify the $R$-homology ribbon cobordisms for various rings $R$.\\

It would be interesting to understand better the interaction between the integer homology \textit{group} and the larger integer homology \textit{monoid}. In the integer homology \textit{group}, because all elements are invertible, all elements are also cancellative: if $[Y] + [Z] = [Z]$ we have $[Y] = 0$. This is not true in an arbitrary monoid, so one might ask if integer homology spheres \textit{remain} cancellative when we pass to the integer homology monoid.

This question can be phrased in terms of the Grothendieck group $\textup{Gr}(\widehat \Theta_{\Bbb Z})$ as follows.

\begin{question}
Does the map $\Theta_{\Bbb Z} \to \textup{Gr}(\widehat{\Theta}_{\Bbb Z})$ have non-trivial kernel? That is, can one find an integer homology 3-sphere $Y$ and a closed oriented 3-manifold $Z$ so that $Y$ is not integer homology cobordant to $S^3$, but $Y \# Z$ is integer homology cobordant to $Z$?
\end{question}

One might imagine that the behavior of integer homology spheres under integer homology cobordism is somehow orthogonal to the behavior of rational homology spheres; very optimistically, one might believe that $j: \Theta_{\Bbb Z} \to \widehat{\Theta}_{\Bbb Z}$ \textit{splits}, meaning that there is a homomorphism $\mu: \widehat{\Theta}_{\Bbb Z} \to \Theta_{\Bbb Z}$ with $\mu j = 1$, and one might then try to understand the structure of the Grothendieck group in terms of $\ker(\mu)$ and $\Theta_{\Bbb Z}$. If such a splitting exists at the level of monoids, we also have such a splitting at the level of groups when we pass to the Grothendieck group.

One way to guarantee this is \textit{impossible}, while also giving an element in the kernel discussed above, is to find indivisible elements in $\Theta_{\Bbb Z}$ which become divisible in $\text{Gr}(\widehat{\Theta}_{\Bbb Z})$. We pose the existence of such homology spheres as a question.

\begin{question}
Is there an integer homology sphere $Y$ with the following properties?
\begin{enumerate}[label=(\roman*)]
    \item $Y$ is indivisible in the integer homology cobordism group: if $n > 1$ and $Y'$ is another integer homology sphere, there is no integer homology cobordism between $Y$ and $\#^n Y'$. 
    \item There \textit{does} exist an $n > 1$, an integer homology sphere $Y'$, and a 3-manifold $Z$ so that $Y \# Z$ is integer homology cobordant to $(\#^n Y') \# Z$. 
\end{enumerate}
\end{question}

Note that if this held, then $-Y \# nY'$ would give an infinite-order element in the kernel of $\Theta_{\Bbb Z} \to \text{Gr}(\widehat{\Theta}_{\Bbb Z})$, answering Question 3 in the positive.

\bibliographystyle{alpha}
\bibliography{main}
\end{document}